\documentclass[a4paper,12pt,twoside]{amsart}
\usepackage[utf8]{inputenc}
\usepackage{latexsym}
\usepackage[pdftex]{graphicx}
\usepackage{hyperref}
\usepackage{amsmath,amssymb, amsthm, amscd}
\usepackage{dsfont}
\usepackage{multicol}
\usepackage[marginratio=1:1,height=597pt,width=455pt,tmargin=117pt]{geometry}
\usepackage[all]{xy}
\usepackage{etoolbox}
\usepackage{enumitem}
\usepackage{tikz-cd}
\usepackage{mathtools}
\usepackage{float}
\usepackage{comment}
\usepackage{array}
\usepackage{bm}
\expandafter\patchcmd\csname\string\proof\endcsname
  {\normalparindent}{0pt }{}{}
\makeatletter
\patchcmd{\@thm}{\thm@headfont{\scshape}}{\thm@headfont{\scshape\bfseries}}{}{}
\patchcmd{\@thm}{\thm@notefont{\fontseries\mddefault\upshape}}{}{}{}
\makeatother
\makeatletter
\patchcmd\@thm
  {\let\thm@indent\indent}{\let\thm@indent\indent}%
  {}{}
\makeatother
\newtheorem*{theorem*}{Theorem}
\newtheorem{theorem}[equation]{Theorem}
\newtheorem{lemma}[equation]{Lemma}

\newtheorem{corollary}[equation]{Corollary}
\theoremstyle{definition}
\newtheorem{definition}[equation]{Definition}
\theoremstyle{definition}
\newtheorem{remark}[equation]{Remark}
\theoremstyle{remark}
\newtheorem{example}[equation]{Example}
\numberwithin{subsection}{section}
\numberwithin{equation}{section}
\newcommand{\iso}{\xrightarrow{
   \,\smash{\raisebox{-0.50ex}{\ensuremath{\scriptstyle\sim}}}\,}}
\title{On irreducible supersingular representations of $\mathrm{GL}_{2}(F)$}
\author{Mihir Sheth}
\address{Department of Mathematics, Indian Institute of Science \\ Bangalore - 560012, India.}
\email{mihirsheth@iisc.ac.in}
\begin{document}
\maketitle

\begin{abstract}
Let $F$ be a non-archimedean local field of residual characteristic $p>3$ and residue degree $f>1$. We study a certain type of diagram, called \emph{cyclic diagrams}, and use them to show that the universal supersingular modules of $\mathrm{GL}_{2}(F)$ admit infinitely many non-isomorphic irreducible admissible quotients.
\end{abstract}

\tableofcontents
\section*{Introduction}
Let $F$ be a non-archimedean local field of residual characteristic $p$ and residue degree $f$. Fix a uniformizer $\varpi\in F$. The theory of smooth representations of reductive $F$-groups on $\overline{\mathbb{F}}_{p}$-vector spaces has its origins in the paper \cite{bl94} of Barthel and Livn\'{e} in which they classify all smooth irreducible representations of $\mathrm{GL}_{2}(F)$ with central characters except \emph{supersingular} representations. The first examples of supersingular representations of $\mathrm{GL}_{2}(F)$ were constructed by Pa\v sk$\bar{\mathrm{u}}$nas using equivariant coefficient systems on the Bruhat-Tits tree, or equivalently, using \emph{diagrams} \cite{pas04}. Let $K$, $Z$ and $N$ denote respectively the standard maximal compact subgroup, the center and the normalizer of the standard Iwahori subgroup $I$ of $\mathrm{GL}_{2}(F)$  so that the stabilizer of the standard vertex of the tree is $KZ$ and that of the standard edge is $N$. A diagram is a finite data of a smooth $KZ$-representation $D_{0}$, a smooth $N$-representation $D_{1}$ and an $IZ$-equivariant map $D_{1}\rightarrow D_{0}$. This data can be glued together (in a non-canonical way) to obtain smooth representations of $\mathrm{GL}_{2}(F)$ inside some injective envelopes. 

In \cite{bp12}, Breuil and Pa\v sk$\bar{\mathrm{u}}$nas develop the theory of diagrams further and construct irreducible supersingular representations of $\mathrm{GL}_{2}(\mathbb{Q}_{p^{f}})$ with prescribed $K$-socles from certain indecomposable (but not irreducible) diagrams. Here, $\mathbb{Q}_{p^{f}}$ is the degree $f$ unramified extension of $\mathbb{Q}_{p}$. Their results, in particular, imply that $\mathrm{GL}_{2}(\mathbb{Q}_{p^{f}})$ with $f>1$ has infinitely many irreducible admissible supersingular representations on which $p$ acts trivially, unlike $\mathrm{GL}_{2}(\mathbb{Q}_{p})$ which has only finitely many such representations. Since the diagrams considered by them are not irreducible, the irreducibility of the corresponding representations of $\mathrm{GL}_{2}(\mathbb{Q}_{p^{f}})$ depends on certain computations with Witt vectors which do not extend to a ramified $F$ or to an $F$ of positive characteristic. In this note, we focus on irreducible diagrams in order to construct irreducible supersingular representations of $\mathrm{GL}_{2}(F)$ for all local fields $F$.

The complexity of supersingular representations of $\mathrm{GL}_{2}(F)$ for $f>1$ can already be seen in the complexity in classifying irreducible diagrams for $f>1$. To this end, we consider a particular type of irreducible diagrams which are rigid enough. We call them \emph{cyclic diagrams}. These are irreducible diagrams on direct sums of extensions of weights such that the action of $\left(\begin{smallmatrix} 0 & 1 \\ \varpi & 0 \end{smallmatrix} \right)$ permutes characters cyclically. We show that cyclic diagrams exist for all $\mathrm{GL}_{2}(F)$ and the $D_{0}$ of any cyclic diagram has more than $2$ irreducible subquotients if $f>1$ (see Theorem \ref{cyclicmodexistence} and Remark \ref{cyclicmodf=1}). As a result, when $f>1$, a family of cyclic diagrams parametrized by $\overline{\mathbb{F}}_{p}^{\times}$ gives rise to infinitely many non-isomorphic irreducible admissible supersingular representations of $\mathrm{GL}_{2}(F)$ with trivial $\varpi$-action (see Theorem \ref{mainthm2}). This implies that, for all local fields $F$ with $f>1$, the universal supersingular modules of $\mathrm{GL}_{2}(F)$ have infinitely many non-isomorphic irreducible admissible quotients (see Corollary \ref{usscor}). While Corollary \ref{usscor} follows from the main results of \cite{bp12} for $F=\mathbb{Q}_{p^{f}}$, it is a new result, to our knowledge, for $F$ ramified over $\mathbb{Q}_{p}$ and for $F$ of positive characteristic.  

We conclude by mentioning a recent note by Z. Wu in the similar spirit in which he gives a uniform proof of the fact that the universal supersingular modules of $\mathrm{GL}_{2}(F)$ are not admissible for any $p$-adic field $F\neq\mathbb{Q}_{p}$ by showing that the supersingular representations are not of finite presentations \cite{wu21}. 
\\
\\
\noindent\textit{Acknowledgments}. This note grew out of discussion meetings with E. Ghate on the topic of diagrams and mod $p$ representations. The author thanks him for these helpful discussions. The author also thanks C. Breuil for his encouraging comment on this note. During the preparation of this note, the author was supported by the Raman Postdoctoral Fellowship of the Indian Institute of Science, Bangalore. 
\\
\\
\textbf{Notation and convention}: Let $p>3$ be a prime number. Let $F$ be a non-archimedean local field of residual characteristic $p$ and residue degree $f$. Let $\mathcal{O}\subseteq F$ be the valuation ring with a uniformizer $\varpi$. Let $\overline{\mathbb{F}}_{p}$ be the algebraic closure of the finite field $\mathbb{F}_{p^{f}}$ of size $p^{f}$. Fix an embedding $\mathbb{F}_{p^{f}}\hookrightarrow\overline{\mathbb{F}}_{p}$. Let $G=\mathrm{GL}_{2}(F)$, $K=\mathrm{GL}_{2}(\mathcal{O})$, $\Gamma=\mathrm{GL}_{2}(\mathbb{F}_{p^{f}})$, and $Z$ be the center of $G$. Let $B$ and $U$ be the subgroups of $\Gamma$ consisting of the upper triangular matrices and the upper triangular unipotent matrices respectively. Let $I$ and $I(1)$ be the preimages of $B$ and $U$ respectively under the reduction modulo $\varpi$ map $K\twoheadrightarrow\Gamma$. The subgroups $I$ and $I(1)$ of $K$ are the Iwahori and the pro-$p$ Iwahori subgroup of $K$ respectively. The normalizer $N$ of $I$ in $G$ is a subgroup generated by $I$ and $\Pi=\left(\begin{smallmatrix}
0 & 1 \\ \varpi & 0 \end{smallmatrix} \right)$. Note that $N$ is also the normalizer of $I(1)$ in $G$. Let $K(1)$ denote the kernel of the map $K\twoheadrightarrow\Gamma$, i.e., first principal congruence subgroup of $K$. Unless stated otherwise, all representations considered in this note are on $\overline{\mathbb{F}}_{p}$-vector spaces. 

A \emph{weight} is an irreducible representation of $\Gamma$. Any weight is of the form of \linebreak $\left(\bigotimes\limits_{j=0}^{f-1}\mathrm{Sym}^{r_{j}}\overline{\mathbb{F}}_{p}^{2}\circ\Phi^{j}\right)\otimes\mathrm{det}^{m}$ for some integers $0\leq r_{0},\ldots,r_{f-1}\leq p-1$ and $0\leq m\leq p^{f}-2$, where $\Phi:\Gamma\rightarrow\Gamma$ is the automorphism induced by the Frobenius map $\alpha\mapsto \alpha^{p}$ on $\mathbb{F}_{p^{f}}$ and $\mathrm{det}:\Gamma\rightarrow\mathbb{F}_{p^{f}}^{\times}$ is the determinant character. We denote such a weight by $\bm{r}\otimes\mathrm{det}^{m}$ where $\bm{r}$ is the $f$-tuple $(r_{0},\ldots,r_{f-1})$ of integers. Let $\sigma=\bm{r}\otimes\mathrm{det}^{m}$ be a weight; its subspace $\sigma^{U}$ of $U$-fixed vectors is $1$-dimensional and stable under the action of $B$ because $B$ normalizes $U$. The resulting $B$-character, denoted by $\chi(\sigma)$, sends $\left(\begin{smallmatrix}a & b \\ 0 & d \end{smallmatrix} \right)\in B$ to $a^{r}(ad)^{m}$ where $r=\sum_{j=0}^{f-1}r_{j}p^{j}$. Any $B$-character valued in $\overline{\mathbb{F}}_{p}^{\times}$ factors through the quotient $B/U$ which is identified with the subgroup of diagonal matrices in $B$ by the section $B/U\rightarrow B$, $\left(\begin{smallmatrix}a & 0 \\ 0 & d \end{smallmatrix} \right)U\mapsto\left(\begin{smallmatrix}a & 0 \\ 0 & d \end{smallmatrix} \right)$. For a $B$-character $\chi$, let $\chi^{s}$ be the inflation to $B$ of the conjugation-by-$s$ character $t\mapsto\chi(sts^{-1})$ on $B/U$ where $s=\left(\begin{smallmatrix}0 & 1 \\ 1 & 0 \end{smallmatrix} \right)$. We say that a weight is \emph{generic} if it is not equal to $(0,0,\ldots,0)\otimes\mathrm{det}^{m}$ and $(p-1,p-1,\ldots,p-1)\otimes\mathrm{det}^{m}$ for any $m$. The map $\sigma\mapsto\chi(\sigma)$ gives a bijection from the set of generic weights to the set of $B$-characters $\chi$ such that $\chi\neq\chi^{s}$. If $\sigma$ is a generic weight, let us denote by $\sigma^{[s]}$ the generic weight corresponding to the character $\chi(\sigma)^{s}$. For $\sigma=\bm{r}\otimes\mathrm{det}^{m}$, $\sigma^{[s]}=(p-1-r_{0},\ldots,p-1-r_{f-1})\otimes\mathrm{det}^{m+r}$. We refer the reader to \cite[\S 1]{bl94} for all non-trivial assertions in this paragraph.

Given two weights $\sigma$ and $\tau$, let $E(\sigma, \tau)$ be the unique non-split $\Gamma$-extension $0\longrightarrow\sigma\longrightarrow E(\sigma, \tau)\longrightarrow\tau\longrightarrow 0$ if it exists \cite[Corollary 5.6]{bp12}. We also denote $E(\sigma,\tau)$ by \begin{tikzcd}\sigma\arrow[r, dash]&\tau\end{tikzcd}. A finite-dimensional representation of $\Gamma$ is said to be \emph{multiplicity-free} if its Jordan-H\"{o}lder factors are multiplicity-free. For any group $H$, the socle and the cosocle of an $H$-representation $\pi$ are denoted by $\mathrm{soc}_{H}\pi$ and $\mathrm{cosoc}_{H}\pi$ respectively.   

Note that a weight is a smooth irreducible representation of $K$ (resp. of $KZ$) and a $B$-character is a smooth $I$-character (resp. $IZ$-character) via the map $K\twoheadrightarrow\Gamma$ (resp. $KZ\twoheadrightarrow\Gamma$). In fact, the weights exhaust all smooth irreducible representations of $K$ (resp. of $KZ$ such that $\varpi$ acts trivially).

\section{Cyclic modules}

We are interested in the following type of representations of $\Gamma$.

\begin{definition} A finite-dimensional representation $D_{0}$ of $\Gamma$ is called a \emph{cyclic module} of $\Gamma$ if there exists a finite set $\lbrace\sigma_{1},\sigma_{2},\ldots,\sigma_{n}\rbrace$ of distinct generic weights such that $E(\sigma_{i},\sigma_{i-1}^{[s]})$ exists for all $1\leq i\leq n$, $D_{0}=\bigoplus_{i=1}^{n} E(\sigma_{i},\sigma_{i-1}^{[s]})$ and $D_{0}^{U}=\bigoplus_{i=1}^{n} E(\sigma_{i},\sigma_{i-1}^{[s]})^{U}=\bigoplus_{i=1}^{n}\chi(\sigma_{i})\oplus\chi(\sigma_{i-1})^{s}$ with the convention $\sigma_{0}=\sigma_{n}$.
\end{definition}

If $D_{0}=\bigoplus_{i=1}^{n} E(\sigma_{i},\sigma_{i-1}^{[s]})$ is a cyclic module of $\Gamma$, then, by Frobenius reciprocity, there is a non-zero map $\mathrm{Ind}^{\Gamma}_{B}\chi(\sigma_{i-1})^{s}\rightarrow E(\sigma_{i},\sigma_{i-1}^{[s]})$ for all $1\leq i\leq n$. Since the principal series representation $\mathrm{Ind}^{\Gamma}_{B}\chi(\sigma_{i-1})^{s}$ has cosocle $\sigma_{i-1}^{[s]}$, and $\sigma_{i}\neq\sigma_{i-1}^{[s]}$, the map $\mathrm{Ind}^{\Gamma}_{B}\chi(\sigma_{i-1})^{s}\rightarrow E(\sigma_{i},\sigma_{i-1}^{[s]})$ is surjective, and hence $\sigma_{i}$ belongs to the first graded piece $\mathrm{gr}^{1}_{\mathrm{cosoc}}\left(\mathrm{Ind}^{\Gamma}_{B}\chi(\sigma_{i-1})^{s}\right)$ of the cosocle filtration of $\mathrm{Ind}^{\Gamma}_{B}\chi(\sigma_{i-1})^{s}$ for all $1\leq i\leq n$. 

\begin{remark}\label{cyclicmodf=1} If $D_{0}=\bigoplus_{i=1}^{n} E(\sigma_{i},\sigma_{i-1}^{[s]})$ is a cyclic module of $\Gamma$ with $n=1$, i.e., $D_{0}=E(\sigma,\sigma^{[s]})$, then the surjective map $\mathrm{Ind}^{\Gamma}_{B}\chi(\sigma)^{s}\rightarrow E(\sigma,\sigma^{[s]})$ is actually an isomorphism: if the kernel is non-zero, then it has socle $\sigma$ because $\mathrm{soc}_{\Gamma}\mathrm{Ind}^{\Gamma}_{B}\chi(\sigma)^{s}=\sigma$. But $\sigma$ also occurs in the image as a subquotient which contradicts the fact that a principal series is multiplicity-free \cite[Lemma 2.2]{bp12}. Therefore $\mathrm{Ind}^{\Gamma}_{B}\chi(\sigma)^{s}\cong E(\sigma,\sigma^{[s]})$, and this forces $\Gamma=\mathrm{GL}_{2}(\mathbb{F}_{p})$ by \cite[Theorem 2.4]{bp12}. 
In fact, any cyclic module of $\mathrm{GL}_{2}(\mathbb{F}_{p})$ is a principal series representation. Indeed, if $\Gamma=\mathrm{GL}_{2}(\mathbb{F}_{p})$ and $E(\sigma,\tau)$ is a non-split $\Gamma$-extension between generic weights $\sigma$ and $\tau$ such that $E(\sigma,\tau)^{U}=\chi(\sigma)\oplus\chi(\tau)$ then $\tau=\sigma^{[s]}$ and thus $E(\sigma,\tau)=\mathrm{Ind}_{B}^{\Gamma}\chi(\sigma)^{s}$ \cite[Corollary 5.6 (i) and Proposition 4.13 or Corollary 14.10]{bp12}.
\end{remark} 

In order to construct cyclic modules of $\Gamma=\mathrm{GL}_{2}(\mathbb{F}_{p^{f}})$ for $f>1$, we take a closer look at the weights appearing in the first graded pieces of cosocle filtrations of principal series. Let $x$ be a formal variable and let $\mathbb{Z}\pm x:=\lbrace n\pm x:n\in\mathbb{Z}\rbrace$ denote the set of linear polynomials in $x$ having integral coefficients with leading coefficient $\pm 1$. Let $\left(\mathbb{Z}\pm x\right)^{f}$ be the set of $f$-tuples of polynomials in $\mathbb{Z}\pm x$. For $\bm\lambda=(\lambda_{0}(x),\ldots,\lambda_{f-1}(x))\in\left(\mathbb{Z}\pm x\right)^{f}$ and $\bm{r}\in\mathbb{Z}^{f}$, let $\bm\lambda(\bm{r}):=\left(\lambda_{0}(r_{0}),\lambda_{1}(r_{1}),\ldots,\lambda_{f-1}(r_{f-1})\right)\in\mathbb{Z}^{f}$. Recall the polynomial $e(\bm\lambda)\in\mathbb{Z}\oplus\bigoplus_{j=0}^{f-1}\mathbb{Z}x_{j}$ associated to $\bm\lambda\in\left(\mathbb{Z}\pm x\right)^{f}$ in \cite[\S 2]{bp12}: 

\[
e(\bm{\lambda})(x_{0},\ldots,x_{f-1}):=
\begin{cases}
\frac{1}{2}\left(\sum\limits_{j=0}^{f-1}p^{j}(x_{j}-\lambda_{j}(x_{j}))\right) & \text{if $\lambda_{f-1}(x_{f-1})\in\lbrace x_{f-1},x_{f-1}-1\rbrace$}, \\
\frac{1}{2}\left(p^{f}-1+\sum\limits_{j=0}^{f-1}p^{j}(x_{j}-\lambda_{j}(x_{j}))\right) & \text{otherwise}.
\end{cases}
\]

For each $f>1$, let $\bm{\mu}\in(\mathbb{Z}\pm x)^{f}$ be the $f$-tuple of polynomials defined by 

\begin{align}
&\mu_{0}(x):=x-1,\nonumber\\&\mu_{1}(x):=p-2-x,\\&\mu_{j}(x):=p-1-x 
\hspace{2mm}\text{for $2\leq j\leq f-1$}.\nonumber
\end{align}
Let $g\in S_{f}$ be the cyclic permutation of an $f$-tuple mapping its $j$-th entry to $(j+1)$-th entry and the last entry to the first one. If $\sigma=\bm\lambda(\bm{r})\otimes\eta$ is a generic weight of $\Gamma=\mathrm{GL}_{2}(\mathbb{F}_{p^{f}})$ for some determinant-power character $\eta$ and $f>1$, then $\mathrm{gr}^{1}_{\mathrm{cosoc}}\left(\mathrm{Ind}^{\Gamma}_{B}\chi(\sigma)^{s}\right)$ consists of $f$ number of weights which can be described by the set 

\[\big\{(g^{i}\bm{\mu})(\bm\lambda(\bm{r}))\otimes\mathrm{det}^{e(g^{i}\bm{\mu})(\bm\lambda(\bm{r}))}\eta:0\leq i\leq f-1\big\}\]
(see \cite[Theorem 2.4]{bp12}).

For $\bm{\lambda}=(\lambda_{0}(x),\ldots,\lambda_{f-1}(x))$ and $\bm{\lambda}'=(\lambda_{0}'(x),\ldots,\lambda_{f-1}'(x))\in(\mathbb{Z}\pm x)^{f}$, let 
\[\bm{\lambda}\circ\bm{\lambda}':=(\lambda_{0}(\lambda_{0}'(x)),\ldots,\lambda_{f-1}(\lambda_{f-1}'(x)))\in(\mathbb{Z}\pm x)^{f}.\]

Define an integer $l$ to be equal to $f$ (resp. $2f$) if $f$ is odd (resp. even). Let  \[\bm\mu^{(0)}:=(x,x,\ldots,x)\hspace{2mm}\text{and}\hspace{2mm}\bm\mu^{(k)}:=g^{k-1}\bm{\mu}\circ g^{k-2}\bm{\mu}\circ\ldots\circ g\bm{\mu}\circ\bm\mu\hspace{2mm} \text{for all $1\leq k\leq l$}.\] 
For $\bm{r}\in\mathbb{Z}^{f}$, let  \[e_{0}(\bm{r}):=0\hspace{2mm}\text{and}\hspace{2mm}
	e_{k}(\bm{r}):=\sum_{j=0}^{k-1}e(g^{j}\bm\mu)(\bm\mu^{(j)}(\bm{r}))\in\mathbb{Z} \hspace{2mm} \text{for all $1\leq k\leq l$}.\]
\begin{lemma}\label{preplemma}
	\begin{enumerate}
\item We have $\bm\mu^{(l)}=\bm\mu^{(0)}=(x,x,\ldots,x)$ in $\left(\mathbb{Z}\pm x\right)^{f}$. 
\item The $f$-tuples $\bm\mu^{(1)},\bm\mu^{(2)},\ldots,\bm\mu^{(l)}$ are all distinct. 
\item The integer $e_{l}(\bm{r})$ is independent of $\bm{r}$ and is $0$ modulo $p^{f}-1$. \end{enumerate} 
\end{lemma}
\begin{proof}
(1) It follows from the definition of $\bm\mu^{(k)}$ that $\bm{\mu}^{(k)}=g^{k-1}\bm{\mu}\circ\bm{\mu}^{(k-1)}$ for all $1\leq k\leq l$. Hence, for $1\leq k\leq l$, 

\begin{equation}\label{inductiveformulae}
\mu_{j}^{(k)}(x)=
\begin{cases}
\mu_{j}^{(k-1)}(x)-1& \text{if $j\equiv 1-k\mod{f}$,}\\
p-2-\mu_{j}^{(k-1)}(x)& \text{if $j\equiv 2-k\mod{f}$,}\\
p-1-\mu_{j}^{(k-1)}(x) & \text{otherwise}.
\end{cases}
\end{equation}  

It is now easy to check using (\ref{inductiveformulae}) that for each $j$, $\mu_{j}^{(l)}(x)=x$.

(2) Let us assign to $\bm\mu^{(k)}$ an element $\bm{m}^{(k)}\in(\mathbb{Z}/2\mathbb{Z})^{f}$ by the rule that its $j$-th entry $m^{(k)}_{j}$ is $0$ if and only if the sign of $x$ in $\mu_{j}^{(k)}(x)$ is $+$. Here, $(\mathbb{Z}/2\mathbb{Z})^{f}$ is the direct sum of $f$ copies of the group $\mathbb{Z}/2\mathbb{Z}$ of order $2$ and has a natural action of $\langle g\rangle$ by group automorphisms. We show that the elements $\bm{m}^{(1)},\bm{m}^{(2)},\ldots,\bm{m}^{(l)}$ are all distinct in $(\mathbb{Z}/2\mathbb{Z})^{f}$ which then implies part (2). We have $\bm{m}^{(1)}=(0,1,1,\ldots,1)$ and $\bm{m}^{(k)}=g^{k-1}\bm{m}^{(1)}+\bm{m}^{(k-1)}$ for $k>1$ because $\bm{\mu}^{(k)}=g^{k-1}\bm{\mu}^{(1)}\circ\bm{\mu}^{(k-1)}$. Suppose $\bm{m}^{(k_{1})}=\bm{m}^{(k_{2})}$ for some $1\leq k_{1}<k_{2}\leq l$. Then  \[\bm{m}^{(k_{1})}=\bm{m}^{(k_{2})}=g^{k_{2}-1}\bm{m}^{(1)}+g^{k_{2}-2}\bm{m}^{(1)}+\ldots+g^{k_{1}}\bm{m}^{(1)}+\bm{m}^{(k_{1})}.\] This gives that  \[g^{k_{1}+(k_{2}-k_{1}-1)}\bm{m}^{(1)}+g^{k_{1}+(k_{2}-k_{1}-2)}\bm{m}^{(1)}+\ldots+g^{k_{1}}\bm{m}^{(1)}=(0,0,\ldots,0).\]  The action of $g^{-k_{1}}$ on both sides then gives  \[g^{k_{2}-k_{1}-1}\bm{m}^{(1)}+g^{k_{2}-k_{1}-2}\bm{m}^{(1)}+\ldots+\bm{m}^{(1)}=(0,0,\ldots,0), \hspace{2mm} \text{i.e.,} \hspace{2mm} \bm{m}^{(k_{2}-k_{1})}=(0,0,\ldots,0).\] This is a contradiction because $k_{2}-k_{1}<l$ and for any $l'<l$, $\bm{m}^{(l')}\neq(0,0,\ldots,0)$. The latter fact can be easily checked by looking at $m^{(l')}_{0}$ and $m^{(l')}_{1}$. One has $m^{(l')}_{0}\neq m^{(l')}_{1}$ for $l'<l$ except when $l=2f$ and $l'=f$ in which case $m^{(l')}_{0}=m^{(l')}_{1}=1$.

(3) Let us first consider $f$ to be odd (so $l=f$). Expanding the expression for $e_{l}(\bm{r})$ and rearranging the terms, one gets \[e_{l}(\bm{r})=c+\sum_{k=1}^{f-1}\mu^{(k)}_{0}(r_{0})+p\sum_{k=0}^{f-2}\mu^{(k)}_{1}(r_{1})+p^{2}\sum_{k=-1}^{f-3}\mu^{(k)}_{2}(r_{2})+\cdots+p^{f-1}\sum_{k=-(f-2)}^{0}\mu^{(k)}_{f-1}(r_{f-1}),\] where $c$ is the constant term of the polynomial $e(g^{f-1}\bm\mu)+e(g^{f-2}\bm\mu)+\cdots+e(g\bm\mu)+e(\bm\mu)$, and $k=-n$ for positive $n$ means $k=f-n$ in the summation $\sum_{k}$. Using (\ref{inductiveformulae}), one checks that each summand $\sum_{k}\mu^{(k)}_{j}(r_{j})$ above (with appropriate lower and upper limit) is independent of $r_{j}$ and equals $\frac{f-1}{2}(p-1)-1$. We leave it to the reader to check that $c\equiv\frac{p^{f}-1}{p-1}\mod{p^{f}-1}$. Therefore, $e_{l}(\bm{r})=\frac{p^{f}-1}{p-1}\left(\frac{f-1}{2}(p-1)\right)\equiv 0\mod{p^{f}-1}$. The proof for even $f$ is similar. In this case, one gets $c\equiv 2\left(\frac{p^{f}-1}{p-1}\right)\mod{p^{f}-1}$, and $\sum_{k}\mu^{(k)}_{j}(r_{j})=2\left(\frac{f-1}{2}(p-1)-1\right)$ for all $j$. Thus, $e_{l}(\bm{r})$ is again $0$ modulo $p^{f}-1$.
\end{proof}

\begin{theorem}\label{cyclicmodexistence}
The group $\Gamma$ admits a multiplicity-free cyclic module $D_{0}$.
\end{theorem}

\begin{proof}
The case $f=1$ is treated in Remark \ref{cyclicmodf=1}. Let $f>1$. The proof is constructive. Start with a weight $\sigma_{0}:=\bm{r}\otimes\eta$ of $\Gamma$ for some $1\leq r_{0},\ldots,r_{f-1}\leq p-3$ and for some determinant-power character $\eta$. Observe that $\sigma_{0}:=\bm\mu^{(0)}(\bm{r})\otimes\mathrm{det}^{e_{0}(\bm{r})}\eta$. Let  \[\sigma_{k}:=\bm\mu^{(k)}(\bm{r})\otimes\mathrm{det}^{e_{k}(\bm{r})}\eta\]  for all $1\leq k\leq l$. We claim that the set $\lbrace\sigma_{1},\sigma_{2},\ldots,\sigma_{l}\rbrace$ is the required set to construct a cyclic module. Using (\ref{inductiveformulae}), one checks that $\mu_{j}^{(k)}(x)\in\{x,x-1,x+1,p-2-x,p-3-x,p-1-x\}$ for all $1\leq k\leq l$ and $0\leq j\leq f-1$. Since $p>3$, this means that the weights $\sigma_{1},\sigma_{2},\ldots,\sigma_{l}$ are well-defined. Further, by Lemma \ref{preplemma} and its proof, one sees that the weights $\sigma_{1},\sigma_{2},\ldots,\sigma_{l}$ are all distinct generic weights and $\sigma_{l}=\sigma_{0}$. Now let $1\leq k\leq l$. We know that the weights appearing in $\mathrm{gr}^{1}_{\mathrm{cosoc}}\left(\mathrm{Ind}^{\Gamma}_{B}\chi(\sigma_{k-1})^{s}\right)$ are 

\[\lbrace(g^{i}\bm\mu)(\bm\mu^{(k-1)}(\bm{r}))\otimes\mathrm{det}^{e(g^{i}\bm\mu)(\bm\mu^{(k-1)}(\bm{r}))}\mathrm{det}^{e_{k-1}(\bm{r})}\eta:0\leq i\leq f-1\rbrace.\]
In particular, $\mathrm{gr}^{1}_{\mathrm{cosoc}}\left(\mathrm{Ind}^{\Gamma}_{B}\chi(\sigma_{k-1})^{s}\right)$ contains the weight 
\[(g^{k-1}\bm\mu)(\bm\mu^{(k-1)}(\bm{r}))\otimes\mathrm{det}^{e(g^{k-1}\bm\mu)(\bm\mu^{(k-1)}(\bm{r}))}\mathrm{det}^{e_{k-1}(\bm{r})}\eta=\bm\mu^{(k)}(\bm{r})\otimes\mathrm{det}^{e_{k}(\bm{r})}\eta=\sigma_{k}.\]
Since $\mathrm{gr}^{0}_{\mathrm{cosoc}}\left(\mathrm{Ind}^{\Gamma}_{B}\chi(\sigma_{k-1})^{s}\right)=\mathrm{cosoc}_{\Gamma}\mathrm{Ind}_{B}^{\Gamma}\chi(\sigma_{k-1})^{s}=\sigma_{k-1}^{[s]}$ is simple, $E(\sigma_{k},\sigma_{k-1}^{[s]})$ exists and is equal to the unique quotient of $\mathrm{Ind}_{B}^{\Gamma}\chi(\sigma_{k-1})^{s}$ with socle $\sigma_{k}$. Since $(\mathrm{Ind}_{B}^{\Gamma}\chi(\sigma_{k-1})^{s})^{U}=\chi(\sigma_{k-1})\oplus\chi(\sigma_{k-1})^{s}$, $E(\sigma_{k},\sigma_{k-1}^{[s]})^{U}=\chi(\sigma_{k})\oplus\chi(\sigma_{k-1})^{s}$. Therefore it follows that $D_{0}:=\bigoplus_{k=1}^{l} E(\sigma_{k},\sigma_{k-1}^{[s]})$ is a cyclic module of $\Gamma$ and has socle of length $l$. 

It remains to show that $D_{0}$ is multiplicity-free. By definition, $\mathrm{soc}_{\Gamma}D_{0}$ is multiplicity-free. Thus also $\sigma_{k_{1}}^{[s]}\neq\sigma_{k_{2}}^{[s]}$ for any $k_{1}\neq k_{2}$, $1\leq k_{1},k_{2}\leq l$, because $(\sigma^{[s]})^{[s]}=\sigma$. Now, if $\sigma_{k_{1}}=\sigma_{k_{2}-1}^{[s]}$ for some $1\leq k_{1},k_{2}\leq l$, then there is a non-split $\Gamma$-extension between $\sigma_{k_{1}}$ and $\sigma_{k_{2}}$. Consider the elements $\bm{m}^{(k_{1})}$ and $\bm{m}^{(k_{2})}$ of $(\mathbb{Z}/2\mathbb{Z})^{f}$ assigned to $\bm{\mu}^{(k_{1})}$ and $\bm{\mu}^{(k_{2})}$ respectively in the proof of Lemma \ref{preplemma}. By \cite[Lemma 5.6(i)]{bp12}, the number of $1$'s in $\bm{m}^{(k_{1})}$ and $\bm{m}^{(k_{2})}$ have different parity. However, if $f$ is odd, then one checks that the number of $1$'s in $\bm{m}^{(k)}$ is always even for all $1\leq k\leq l$ implying that $\sigma_{k_{1}}\neq\sigma_{k_{2}-1}^{[s]}$ for any $1\leq k_{1},k_{2}\leq l$. If $f$ is even, then it is not true that the number of $1$'s in $\bm{m}^{(k)}$ is always either even or odd, and it is a priori possible that $\sigma_{k}=\sigma_{k+f}^{[s]}$ because $\bm{m}^{(k)}+\bm{m}^{(k+f)}=(1,1,\ldots,1)$. However, using (\ref{inductiveformulae}), one explicitly checks that $\sigma_{k}\neq\sigma_{k+f}^{[s]}$ for any $1\leq k\leq l$.    
\end{proof}

\begin{remark} When $f$ is odd, the argument given in the proof Theorem of \ref{cyclicmodexistence} shows that any cyclic module of $\Gamma$ is multiplicity-free. This is not true when $f$ is even (see the next remark). 
We further point out that the definition of the $f$-tuple $\bm{\mu}$ is not canonical. Any other cyclic permutation of $\bm{\mu}$ also gives rise to a cyclic module of $\Gamma$ by the same construction as above. We expect that all multiplicity-free cyclic modules of $\Gamma$ are obtained in this way, and thus any multiplicity-free cyclic module of $\Gamma$ has socle of length $l$. 
\end{remark}

\begin{example}\label{examples} The construction in the proof of Theorem \ref{cyclicmodexistence} produces following multiplicity-free cyclic modules for $f=2,3$. The weights are written without their twists by determinant-power characters.
\begin{itemize}
\item[$f=2$:]$D_{0}=$
\begin{tikzcd}(r_{0}-1,p-2-r_{1})\arrow[r, dash]&(p-1-r_{0},p-1-r_{1})\hspace{2mm}\bigoplus 
\end{tikzcd}
\\
\text{\hspace{10mm}}\begin{tikzcd}(p-1-r_{0},p-3-r_{1})\arrow[r, dash]&(p-r_{0},r_{1}+1)\hspace{2mm}\bigoplus 
\end{tikzcd}
\\
\text{\hspace{10mm}}\begin{tikzcd}(p-2-r_{0},r_{1}+1)\arrow[r, dash]&(r_{0},r_{1}+2)\hspace{2mm}\bigoplus 
\end{tikzcd}
\\
\text{\hspace{10mm}}\begin{tikzcd}(r_{0},r_{1})\arrow[r, dash]&(r_{0}+1,p-2-r_{1}).\hspace{2mm} 
\end{tikzcd}
\item[$f=3$:]$D_{0}=$
\begin{tikzcd}(r_{0}-1,p-2-r_{1},p-1-r_{2})\arrow[r, dash]&(p-1-r_{0},p-1-r_{1},p-1-r_{2})\hspace{2mm}\bigoplus 
\end{tikzcd}
\\
\text{\hspace{10mm}}\begin{tikzcd}(p-1-r_{0},r_{1}+1,p-2-r_{2})\arrow[r, dash]&(p-r_{0},r_{1}+1,r_{2})\hspace{2mm}\bigoplus 
\end{tikzcd}
\\
\text{\hspace{10mm}}\begin{tikzcd}(r_{0},r_{1},r_{2})\arrow[r, dash]&(r_{0},p-2-r_{1},r_{2}+1).\hspace{2mm} 
\end{tikzcd}
\end{itemize}

\end{example}

\begin{remark} Let $\mathbb{Q}_{p^{f}}$ denote the degree $f$ unramified extension of $\mathbb{Q}_{p}$. The multiplicity-free cyclic module of $\mathrm{GL}_{2}(\mathbb{F}_{p^{2}})$ (resp. of $\mathrm{GL}_{2}(\mathbb{F}_{p^{3}})$) in Example \ref{examples} occurs as a submodule of $D_{0}(\rho)$ of a Diamond diagram associated to an irreducible (resp. reducible split) generic Galois representation $\rho$ of $\mathbb{Q}_{p^{2}}$ (resp. of $\mathbb{Q}_{p^{3}}$) (see \cite[\S 14]{bp12}). 

In \cite{sch22}, M. Schein constructs irreducible supersingular representations of $G=\mathrm{GL}_{2}(F)$ with $K$-socles compatible with Serre's weight conjecture for a ramified $p$-adic field $F$ of residue degree $2$. His construction is based on constructing cyclic modules of $\mathrm{GL}_{2}(\mathbb{F}_{p^{2}})$ with prescribed socles. The involved cyclic modules of $\mathrm{GL}_{2}(\mathbb{F}_{p^{2}})$ have socles of lengths $>l$ and are not multiplicity-free (see \cite[Example 3.9]{sch22}).
\end{remark}

\section{Cyclic diagrams}

Recall from \cite[\S 9]{bp12} that a \emph{diagram} (of $G$) is a data $(D_{0},D_{1},r)$ consisting of a smooth $KZ$-representation $D_{0}$, a smooth $N$-representation $D_{1}$ and an $IZ$-equivariant map $r:D_{1}\rightarrow D_{0}$. A diagram $(D_{0},D_{1},r)$ is called a \emph{basic $0$-diagram} if $\varpi$ acts trivially on $D_{0}$ and $D_{1}$, and the map $r$ induces an isomorphism $D_{1}\cong D_{0}^{I(1)}$ of $IZ$-representations. Now, let $D_{0}=\bigoplus_{i=1}^{n}E(\sigma_{i},\sigma_{i-1}^{[s]})$ be a multiplicity-free cyclic module of $\Gamma$. Viewing $D_{0}$ as a smooth $KZ$-representation via $KZ\twoheadrightarrow\Gamma$ with trivial $\varpi$-action, $D_{1}:=D_{0}^{I(1)}=D_{0}^{U}$ can be equipped with a smooth $N$-action by defining the action of $\Pi:\chi(\sigma_{i})\rightarrow\chi(\sigma_{i})^{s}$ to be the multiplication by a scalar $t_{i}\in\overline{\mathbb{F}}_{p}^{\times}$ for all $i$ after choosing bases. This defines a unique $N$-action on $D_{1}$ such that $\varpi$-acts trivially and gives a basic $0$-diagram $(D_{0},D_{1},\mathrm{can})$ where $\mathrm{can}:D_{1}\hookrightarrow D_{0}$ is the canonical inclusion.

\begin{definition} A basic $0$-diagram $(D_{0},D_{1},\mathrm{can})$ on a multiplicity-free cyclic module $D_{0}$ is called a \emph{cyclic diagram}.
\end{definition} 

Note that a cyclic diagram exists for all $G$ by Theorem \ref{cyclicmodexistence}.

\begin{lemma}\label{sslemma} Let $(D_{0},D_{1},\mathrm{can})$ be a cyclic diagram on a cyclic module $D_{0}=\bigoplus_{i=1}^{n}E(\sigma_{i},\sigma_{i-1}^{[s]})$ and let $\Pi:\chi(\sigma_{i})\rightarrow\chi(\sigma_{i})^{s}$ be given by the multiplication by scalar $t_{i}\in\overline{\mathbb{F}}_{p}^{\times}$ for all $1\leq i\leq n$. Then  
\begin{enumerate}
\item $(D_{0},D_{1},\mathrm{can})$ is irreducible, and
\item the isomorphism class of $(D_{0},D_{1},\mathrm{can})$ is determined by the product $t_{1}t_{2}\ldots t_{n}\in\overline{\mathbb{F}}_{p}^{\times}$.
\end{enumerate}
\end{lemma}
\begin{proof}
(1) Let $V\subseteq D_{0}$ be a non-zero $KZ$-subrepresentation such that $V^{I(1)}$ is stable under the action of $N$. Then, for some $1\leq i\leq n$, $V$ contains $\sigma_{i}$ and thus also contains $\chi(\sigma_{i})$. Since $\Pi(\chi(\sigma_{i}))=\chi(\sigma_{i})^{s}$, $V$ contains $\chi(\sigma_{i})^{s}$. By Frobenius reciprocity, there is a non-zero map $\mathrm{Ind}_{B}^{\Gamma}\chi(\sigma_{i})^{s}\rightarrow V$ whose image is $E(\sigma_{i+1},\sigma_{i}^{[s]})$. Thus $E(\sigma_{i+1},\sigma_{i}^{[s]})\subseteq V$. Continuing in this way, we get that $D_{0}=\bigoplus_{i=1}^{n}E(\sigma_{i},\sigma_{i-1}^{[s]})\subseteq V$. Hence $V=D_{0}$. 

(2) Let $D=(D_{0},D_{1},\mathrm{can})$ and let $D'$ be a diagram isomorphic to $D$. Then $D'$ is also a cyclic diagram on the cyclic module $D_{0}$. Let $\Pi:\chi(\sigma_{i})\rightarrow\chi(\sigma_{i})^{s}$ in $D'$ be given by the multiplication by scalar $t_{i}'\in\overline{\mathbb{F}}_{p}^{\times}$ for all $1\leq i\leq n$. As the diagrams $D$ and $D'$ are isomorphic, there is an isomorphism $\varphi:D_{0}\rightarrow D_{0}$ of $KZ$-representations such that $\varphi(\Pi v)=\Pi\varphi(v)$ for all $v\in D_{1}$. Since $D_{0}$ is multiplicity-free, an easy application of Schur's lemma gives 
\[\mathrm{End}_{KZ}(D_{0})=\mathrm{End}_{\Gamma}(D_{0})\cong\mathrm{End}_{\Gamma}(E(\sigma_{1},\sigma_{n}^{[s]}))\times\ldots\times\mathrm{End}_{\Gamma}(E(\sigma_{n},\sigma_{n-1}^{[s]}))\cong\overline{\mathbb{F}}_{p}^{n}.\]
So, if the isomorphism $\varphi$ corresponds to $(a_{1},\ldots,a_{n})\in\left(\overline{\mathbb{F}}_{p}^{\times}\right)^{n}$, then $(a_{1},\ldots,a_{n})$ satisfies \linebreak $a_{i}=a_{i-1}t_{i-1}'(t_{i-1})^{-1}$ for all $1\leq i\leq n$. This implies that $t_{1}t_{2}\ldots t_{n}=t_{1}'t_{2}'\ldots t_{n}'$. On the other hand, if $t_{1}t_{2}\ldots t_{n}=t_{1}'t_{2}'\ldots t_{n}'$, then the scalar multiplications by $a_{i}=\prod_{j=1}^{i-1}t_{j}'t_{j}^{-1}$ on $E(\sigma_{i},\sigma_{i-1}^{[s]})$ with $a_{1}=1$ give an isomorphism of cyclic diagrams on $D_{0}$. See also \cite[Proposition 4.4]{dl21}.
\end{proof}

For a cyclic diagram $D=(D_{0},D_{1},\mathrm{can})$, we introduce the notation $t(D)=t_{1}t_{2}\ldots t_{n}$ for later use. With this notation, Lemma \ref{sslemma} (2) says that the map $D\mapsto t(D)$ gives a bijection between the set of isomorphism classes of cyclic diagrams on $D_{0}$ and $\overline{\mathbb{F}}_{p}^{\times}$.

\section{Supersingular representations}
We now use cyclic diagrams to show that $G=\mathrm{GL}_{2}(F)$ admits infinitely many smooth admissible irreducible supersingular representations when $F$ has residue degree $f>1$. It uses the following key theorem of Breuil and Pa\v sk$\bar{\mathrm{u}}$nas.

\begin{theorem}\label{existencethm}
Let $(D_{0},D_{1},r)$ be a basic $0$-diagram such that $D_{0}^{K(1)}$ is finite-dimensional. Then there exists a smooth admissible representation $\pi$ of $G$ on which $\varpi$ acts trivially, and such that
\begin{enumerate}
\item one has the inclusion $(D_{0},D_{1},r)\subseteq(\pi|_{KZ},\pi|_{N},\mathrm{id})$ of diagrams, 
\item $\pi$ is generated as a $G$-representation by $D_{0}$, and
\item $\mathrm{soc}_{\Gamma}D_{0}=\mathrm{soc}_{K}D_{0}=\mathrm{soc}_{K}\pi$.
\end{enumerate}
Moreover, if $(D_{0},D_{1},r)$ is irreducible, then any such $G$-representation $\pi$ is irreducible.
\end{theorem}
\begin{proof}
The first part is essentially proved in \cite[Theorem 9.8]{bp12}. See also the proof of \cite[Theorem 19.8 (i)]{bp12}. The proof of the irreducibility of $\pi$ is given in unpublished lecture notes of Breuil \cite[Proposition 5.11]{bre07}. We reproduce it here: let $\pi'\subseteq\pi$ be a nonzero $G$-subrepresentation. Then $\pi'\cap D_{0}$ is a non-zero $KZ$-subrepresentation of $D_{0}$ by (3), and $(\pi'\cap D_{0})^{I(1)}=\pi'\cap D_{1}$ is stable under the action of $\Pi$. Hence $(\pi'\cap D_{0},(\pi'\cap D_{0})^{I(1)},\mathrm{can})$ is a non-zero subdiagram of $(D_{0},D_{1},r)$. By irreducibility of $(D_{0},D_{1},r)$, we get $\pi'\cap D_{0}=D_{0}$. Hence, $\pi'=\pi$ using (2). 
\end{proof}

When $F$ has residue degree $1$, the cyclic diagrams are the basic $0$-diagrams on principal series representations of $\mathrm{GL}_{2}(\mathbb{F}_{p})$ (Remark \ref{cyclicmodexistence}) and thus Theorem \ref{existencethm} applied to cyclic diagrams gives rise to irreducible (ramified) principal series representations of $G$ (see \cite[\S 10]{bp12}). In contrast, when $F$ has residue degree $f>1$, Theorem \ref{existencethm} applied to cyclic diagrams gives rise to irreducible supersingular representations of $G$ as we shall see now. Recall from \cite{bl94} that a smooth irreducible representation $\pi$ of $G$ with central character is a quotient of $\pi(\sigma,\lambda,\chi):=\frac{\mathrm{ind}^{G}_{KZ}\sigma}{(T-\lambda)}\otimes(\chi\circ\mathrm{det})$ for some weight $\sigma$, some $\lambda\in\overline{\mathbb{F}}_{p}^{\times}$ and some smooth character $\chi:F^{\times}\rightarrow\overline{\mathbb{F}}_{p}^{\times}$. Here, $\mathrm{ind}^{G}_{KZ}\sigma$ is the compactly induced representation with $\varpi$ acting trivially on $\sigma$, and $T\in\mathrm{End}_{G}(\mathrm{ind}^{G}_{KZ}\sigma)$ is the distinguished Hecke operator. By definition, $\pi$ is supersingular if it is a quotient of some $\pi(\sigma,0,\chi)$. The representations $\pi(\sigma,0,\chi)$ are called \emph{universal supersingular modules}. 

\begin{theorem}\label{mainthm2}
Let $F$ be a non-archimedean local field of residue degree $f>1$. Then the group $G$ admits infinitely many non-isomorphic smooth admissible irreducible supersingular representations on which $\varpi$ acts trivially. Further, all these representations have the same $K$-socle.
\end{theorem}
\begin{proof}

We use the existence of multiplicity-free cyclic modules from Theorem \ref{cyclicmodexistence} to construct a family of cyclic diagrams of $G$. Let $D_{0}$ be a multiplicity-free cyclic module constructed in Theorem \ref{cyclicmodexistence} and for each $t\in\overline{\mathbb{F}}_{p}^{\times}$, let $D(t)=(D_{0},D_{1},\mathrm{can})$ be a cyclic diagram on $D_{0}$ such that $t(D(t))=t$. By Theorem \ref{existencethm}, there is a smooth admissible representation $\pi(t)$ (fix one for each $D(t)$) of $G$ with trivial action of $\varpi$ such that $D(t)\subseteq(\pi(t)|_{KZ},\pi(t)|_{N},\mathrm{can})$, $D_{0}$ generates $\pi(t)$ as a $G$-representation, and $\mathrm{soc}_{K}D_{0}=\mathrm{soc}_{K}\pi(t)$. We claim that $\lbrace\pi(t)\rbrace_{t\in\overline{\mathbb{F}}_{p}^{\times}}$ is the desired family of representations of $G$. By Lemma \ref{sslemma} (1) and Theorem \ref{existencethm}, each $\pi(t)$ is an irreducible $G$-representation. 

Suppose there is an isomorphism $\varphi:\pi(t)\iso\pi(t')$ of $G$-representations for $t\neq t'$. It restricts to an isomorphism $\varphi:D_{0}\iso D_{0}$ of $KZ$-representations because $\mathrm{soc}_{K}D_{0}=\mathrm{soc}_{K}\pi(t)=\mathrm{soc}_{K}\pi(t')$ and because $D_{0}$ is multiplicity-free. This gives rise to an isomorphism $D(t)\cong D(t')$ of cyclic diagrams which contradicts Lemma \ref{sslemma} (2). Thus $\pi(t)$ and $\pi(t')$ are not isomorphic for $t\neq t'$.

It remains to show that each $\pi(t)$ is supersingular. Let $\sigma\in\mathrm{soc}_{K}\pi(t)$. Then 
\linebreak $\mathrm{Hom}_{G}(\mathrm{ind}_{KZ}^{G}\sigma,\pi(t))=\mathrm{Hom}_{K}(\sigma,\pi(t)^{K(1)})$ is a non-zero finite-dimensional $\overline{\mathbb{F}}_{p}$-vector space because $\pi(t)$ is admissible. Hence $\mathrm{Hom}_{G}(\mathrm{ind}_{KZ}^{G}\sigma,\pi(t))$ contains a non-zero eigenvector for the action of Hecke operator $T$ with eigenvalue, let's say, $\lambda$. As $\pi(t)$ is irreducible, it follows that $\pi(t)$ is a quotient of $\pi(\sigma,\lambda,1)$. If $\lambda\neq 0$, then by \cite[Lemma 28 and Theorem 33]{bl94} we have $\mathrm{dim}_{\overline{\mathbb{F}}_{p}}\pi(t)^{I(1)}\leq 2$. However, as $f>1$, $\mathrm{soc}_{K}D_{0}$ is not irreducible and thus $\mathrm{dim}_{\overline{\mathbb{F}}_{p}}D_{0}^{I(1)}=\mathrm{dim}_{\overline{\mathbb{F}}_{p}}D_{1}\geq 4$ (in fact, $\mathrm{dim}_{\overline{\mathbb{F}}_{p}}D_{1}=2l$). But this implies that $\mathrm{dim}_{\overline{\mathbb{F}}_{p}}\pi(t)^{I(1)}>2$ because $\pi(t)$ contains $D_{0}$. So we get a contradiction. Therefore $\lambda=0$ and $\pi({t})$ is supersingular.  
\end{proof}

Recall from \cite[Corollary 31]{bl94} that $\pi(\sigma,\lambda,\chi)$ has a unique (admissible) irreducible quotient for $\lambda\neq 0$. However for $\lambda=0$, we have the following result as an immediate corollary of Theorem \ref{mainthm2}:

\begin{corollary}\label{usscor} Let $F$ be a non-archimedean local field of residue degree $f>1$. Then the universal supersingular module $\pi(\sigma,0,\chi)$ of $G$ has infinitely many non-isomorphic admissible irreducible quotients for any given weight $\sigma=\bm{r}\otimes\eta$ with $1\leq r_{0},\ldots,r_{f-1}\leq p-3$ and any smooth character $\chi$.  
\end{corollary}
\begin{proof}
As in the proof of Theorem \ref{mainthm2}, consider a family $\lbrace D(t)\rbrace_{t\in\overline{\mathbb{F}}_{p}^{\times}}$ of cyclic diagrams on a cyclic module $D_{0}$ from Theorem \ref{cyclicmodexistence} whose socle contains the given weight $\sigma$, and let $\lbrace \pi(t)\rbrace_{t\in\overline{\mathbb{F}}_{p}^{\times}}$ be a corresponding family of smooth admissible irreducible supersingular $G$-representations. By the proof of Theorem \ref{mainthm2}, each $\pi(t)$ occurs as a quotient of $\pi(\sigma,0,1)$. So the corollary holds for $\pi(\sigma,0,1)$ and hence also for its smooth twist $\pi(\sigma,0,\chi)$. 
\end{proof}

\begin{remark} If $F=\mathbb{Q}_{p^{f}}$ with $f>1$, then the recent works of Le \cite{le19} and Ghate-Sheth \cite{gs20} show that the universal supersingular modules of $G$ also admit infinitely many non-isomorphic non-admissible irreducible quotients. 
\end{remark}

\end{document}